\documentclass{amsart}
\usepackage[T1]{fontenc}
\usepackage{amstext}
\usepackage{amsthm}
\usepackage{amssymb}
\usepackage[unicode=true,pdfusetitle,
 bookmarks=true,bookmarksnumbered=false,bookmarksopen=false,
 breaklinks=false,pdfborder={0 0 0},pdfborderstyle={},backref=false,colorlinks=false]
 {hyperref}

\makeatletter
\numberwithin{equation}{section}
\numberwithin{figure}{section}
\theoremstyle{plain}
\newtheorem{thm}{\protect\theoremname}[section]
\theoremstyle{remark}
\newtheorem{rem}[thm]{\protect\remarkname}
\theoremstyle{plain}
\newtheorem{prop}[thm]{\protect\propositionname}
\theoremstyle{plain}
\newtheorem{cor}[thm]{\protect\corollaryname}
\ifx\proof\undefined
\newenvironment{proof}[1][\protect\proofname]{\par
	\normalfont\topsep6\p@\@plus6\p@\relax
	\trivlist
	\itemindent\parindent
	\item[\hskip\labelsep\scshape #1]\ignorespaces
}{%
	\endtrivlist\@endpefalse
}
\providecommand{\proofname}{Proof}
\fi

\usepackage{tikz-cd}
\usepackage{mathtools}
\usepackage{enumitem}
\tikzcdset{scale cd/.style={every label/.append style={scale=#1},
    cells={nodes={scale=#1}}}}
\usepackage{eucal}

\theoremstyle{definition}
\newtheorem{cst}[thm]{Construction}

\makeatother

\providecommand{\corollaryname}{Corollary}
\providecommand{\propositionname}{Proposition}
\providecommand{\remarkname}{Remark}
\providecommand{\theoremname}{Theorem}

\begin{document}
\global\long\def\sf#1{\mathsf{#1}}%

\global\long\def\cal#1{\mathcal{#1}}%

\global\long\def\bb#1{\mathbb{#1}}%

\global\long\def\fr#1{\mathfrak{#1}}%

\global\long\def\o#1{\overline{#1}}%

\global\long\def\pr#1{\left(#1\right)}%

\global\long\def\opn#1{\operatorname{#1}}%

\global\long\def\SS{\sf{sSet}}%

\global\long\def\Cat{\sf{Cat}}%

\global\long\def\SG{\sf{sGrp}}%

\global\long\def\Del{\mathbf{\Delta}}%

\global\long\def\id{\operatorname{id}}%

\global\long\def\op{\mathrm{op}}%

\global\long\def\ot{\leftarrow}%

\global\long\def\ho{\operatorname{ho}}%

\global\long\def\sk{\operatorname{sk}}%

\global\long\def\adj{\stackrel[\longleftarrow]{\longrightarrow}{\bot}}%

\global\long\def\W{\overline{W}}%

\title{Classifying Space via Homotopy Coherent Nerve}
\author{Kensuke Arakawa}
\email{arakawa.kensuke.22c@st.kyoto-u.ac.jp}
\address{Department of Mathematics, 
Kyoto University, 
Kyoto, 
606-8502, 
Japan}
\subjclass{55R35, 18N60}
\keywords{classifying space, homotopy coherent nerve.}
\begin{abstract}
We prove that the classifying space of a simplicial group is modeled
by its homotopy coherent nerve. We will also show that the claim remains
valid for simplicial groupoids.
\end{abstract}
\maketitle

\section{Introduction}

Classically, the classifying space of a topological group $G$ is
defined to be the base space of a principal $G$-bundle with weakly
contractible total space. There is a parallel construction for simplicial
groups: If $G$ is a simplicial group, then a \textbf{principal $G$-fibration}
is a map $p\colon E\to B$ of simplicial sets, such that $E$ is a right
$G$-simplicial set, each $E_{n}$ is $G_{n}$-free, and $E/G\cong B$.
The \textbf{classifying space} $BG$ is defined to be the base of
a principal $G$-fibration whose total is a contractible Kan complex. 

Just as the topological classifying space can be constructed by regarding
a group as a category with one object, taking its nerve, and then
applying the realization functor, it is expected that a similar result holds
for simplicial groups. That is, given a simplicial group $G$, we
might want to guess that the homotopy coherent nerve $NG$, where
$G$ is regarded as a simplicial category with one object, has the
homotopy type of $BG$. The purpose of this note is to prove this
result and a further generalization to the case where $G$ is a simplicial
groupoid.

We note that the statement of the main result of this note (Theorem
\ref{thm:main}) appeared in \cite{Hinich07} previously; however,
Hinich's argument does not seem to hold: He computes the homotopy
groups of $NG$ and $BG$, and constructs a comparison map $BG\to NG$,
but does not explain why the map induces isomorphisms in the homotopy
groups. Fortunately, this gap was recently filled by Minichiello,
Rivera, and Zeinalian in \cite{Rivera_22}. We will
follow an alternative path: Instead of comparing $BG$ and $NG$,
we will compare the corresponding left adjoints. This key insight,
which was communicated to us by Dmitri Pavlov in \cite{426162}, leads
to a more direct and concise argument.

\subsection*{Notation and Terminology}

By a \textbf{simplicial category}, we mean a simplicially enriched
category. If $\cal C$ is a simplicial category and $x,y$ are its
objects, we will write $\cal C\pr{x,y}$ for the simplicial set of
maps from $x$ to $y$ and call its $n$-simplex an \textbf{$n$-arrow}
from $x$ to $y$. The ordinary category consisting of the $n$-arrows
is denoted by $\cal C_{n}$. A \textbf{simplicial groupoid} is a simplicial
category whose $n$-arrows are all invertible, for any $n\geq0$.
A \textbf{simplicial group} $G$, i.e., a simplicial object in the
category of groups, will be identified with a simplicial groupoid
with a single object whose endo-simplicial set is $G$.

A simplicial set is said to be \textbf{reduced} if it has only one
vertex.

We will write $\Cat_{\Delta}$, $\sf{Grpd}_{\Delta}$, $\SG$, $\SS$,
$\SS_{0}$, $\Del$ for the categories of small simplicial categories,
its full subcategory of simplicial groupoids, simplicial groups, simplicial
sets, reduced simplicial sets, and the simplex category (the category
of nonempty finite ordinals and poset map) respectively. Unless stated
otherwise, $\Cat_{\Delta}$ and $\SS$ carry the Bergner and Kan-Quillen
model structures, respectively.

The \textbf{homotopy coherent nerve} $N\colon\Cat_{\Delta}\to\SS_{\mathrm{Joyal}}$,
originally due to Cordier \cite{Cordier_hcnerve}, is a right Quillen
equivalence from the Bergner model structure to the Joyal model structure
arising from a simplicial object in $\Cat_{\Delta}$. There are two
conventions for the choice of a simplicial object. The first one uses
the cosimplicial simplicial category $\fr C[\Delta^{\bullet}]$. The
simplicial category $\fr C[\Delta^{n}]$ has as its objects the integers
$0,\dots,n$, and its mapping simplicial sets are given by
\[
\fr C[\Delta^{n}]\pr{i,j}=N\pr{P_{i,j}},
\]
where $P_{i,j}$ is the poset of subset of subsets $I\subset[n]$
with minimal element $i$ and maximal element $j$, ordered by inclusion.
The other one uses the cosimplicial simplicial category $\widetilde{\fr C}[\Delta^{\bullet}]$,
where $\widetilde{\fr C}[\Delta^{n}]$ is obtained from $\fr C[\Delta^{n}]$
by taking the opposites of the mapping simplicial sets. We will opt
for the latter convention because it makes our exposition more concise\footnote{The choice of the convention is insignificant. Indeed, since the geometric
realization of a simplicial set is naturally homeomorphic to that
of the opposite, if $\mathcal{C}$ is a topological category (i.e.,
a category enriched over the category of compactly generated weak
Hausdorff spaces), then we have a natural bijection
\[
\Cat_{\Delta}\pr{\fr C[\Delta^{\bullet}],\opn{Sing}\cal C}\cong\Cat_{\Delta}\pr{\widetilde{\fr C}[\Delta^{\bullet}],\opn{Sing}\cal C}.
\]
Combining this with the fact that the functor $\opn{Sing}\left|-\right|\colon\Cat_{\Delta}\to\Cat_{\Delta}$
admits a natural weak equivalence from the identity, we find that
the two homotopy coherent nerve functors arising from different conventions
can be joined by a zig-zag of natural transformations whose components
at fibrant simplicial categories are weak categorical equivalences.}. For a comprehensive account of the homotopy coherent nerve functor
and the model structures of Bergner and Joyal, we refer the reader
to \cite[\S 1.1.5, \S 2.2.5, \S A.3.2]{LurieHTT}. But beware that
in \cite{LurieHTT}, Lurie adopts the first convention for the homotopy
coherent nerve functor.

The term ``\textbf{$\infty$-category}'' is a synonym for ``quasi-category''
in the sense of Joyal \cite{Joyal_qcat_Kan}. The term ``\textbf{$\infty$-groupoid}''
will be used as a synonym for ``Kan complex.''

\subsection*{Acknowledgment}

As mentioned above, the underlying idea of this note is entirely due
to \cite{Hinich07} and Dmitri Pavlov's answer given in \cite{426162}.
The author especially thanks Dmitri Pavlov for his hospitality and
patience, and for commenting on an earlier draft of this note. The
author would also like to express gratitude to his advisor Daisuke
Kishimoto, who is always willing to go out of his way to help and who suggested making this note public in the first place. Finally,
the author is grateful to Daisuke Kishimoto and Mitsunobu Tsutaya
for reading earlier drafts, spotting errors, and making helpful suggestions. 

\section{Review of Simplicial Classifying Spaces}

In this section, we review some basic results and constructions on
simplicial groups which we use freely in the next section. 

One of the guiding principles in higher category theory is Grothendieck's
homotopy hypothesis, which states that ``spaces'' and ``higher
groupoids'' should be the same. The $\W$-construction, which we
now introduce, provides an incarnation of this principle.

\begin{cst}Let $\cal G$ be a simplicial groupoid. We define a simplicial
set $\W\cal G$ as follows: An $n$-simplex of $\W\cal G$ is a sequence
\[
x_{n}\xleftarrow{g_{0}}x_{n-1}\xleftarrow{g_{1}}\cdots\xleftarrow{g_{n-2}}x_{1}\xleftarrow{g_{n-1}}x_{0},
\]
where $x_{0},\dots,x_{n}$ are objects of $\cal G$ and $g_{n-i}\colon x_{i-1}\to x_{i}$
is an $\pr{n-i}$-arrow in $\cal G$. For $n\ge1$, the face map $d_{i}\colon \pr{\o W\cal G}_{n}\to\pr{\W\cal G}_{n-1}$
is given by
\[
d_{i}\pr{g_{0},\dots,g_{n-1}}=\begin{cases}
\pr{g_{0},\dots,g_{n-2}} & \text{if }i=0,\\
(g_{0},\dots,g_{n-i-2},\\
\;\;\;g_{n-i-1}\circ d_{0}g_{n-i},\dots,d_{i-2}g_{n-2},d_{i-1}g_{n-1}) & \text{if }0<i<n,\\
\pr{d_{1}g_{1},\dots,d_{n-1}g_{n-1}} & \text{if }i=n,
\end{cases}
\]
while for $n\geq0$, the degeneracy map $s_{i}\colon \pr{\W\cal G}_{n}\to\pr{\W\cal G}_{n+1}$
is given by
\[
s_{i}\pr{g_{0},\dots,g_{n-1}}=\begin{cases}
\pr{g_{0},\dots,g_{n-1},\id} & \text{if }i=0,\\
\pr{g_{0},\dots,g_{n-i-1},\id,s_{0}g_{n-i},\dots,s_{i-1}g_{n-1}} & \text{if }0<i<n,\\
\pr{\id,s_{0}g_{0},\cdots,s_{n-1}g_{n-1}} & \text{if }i=n.
\end{cases}
\]

\end{cst}
\begin{rem}
Our definition of the functor $\W\colon \sf{Grpd}_{\Delta}\to\SS$ is the
opposite of that defined in \cite[Chapter V, \S7]{GoerssJardine},
in the sense that they instead consider the functor $\sf{Grpd}_{\Delta}\xrightarrow{\pr -^{\op}}\sf{Grpd}_{\Delta}\xrightarrow{\o W}\SS$.
The discrepancy arose from the fact that we consider \textit{right}
actions of simplicial groups, whereas in \cite{GoerssJardine} simplicial
groups act from the \textit{left}. 
\end{rem}
\begin{prop}
[{\cite[Theorems 7.6, 7.8]{GoerssJardine}}]The following specifications
determine a model structure on $\sf{sGrpd}$:
\begin{itemize}
\item A map $f\colon \cal G\to\cal H$ is a weak equivalence if and only if the
following conditions are satisfied:
\begin{itemize}
\item The functor $f\colon \cal G_{0}\to\cal H_{0}$ induces a bijection between
the sets of components of $\cal G_{0}$ and $\cal H_{0}$.
\item For each object $x\in\cal G$, the map $\cal G\pr{x,x}\to\cal H\pr{fx,fx}$
is a weak homotopy equivalence.
\end{itemize}
\item A map $f\colon \cal G\to\cal H$ is a fibration if and only if the following
conditions are satisfied:
\begin{itemize}
\item Given a morphism $v\colon y\to y'$ in $\cal H_{0}$ and an object $x\in\cal G$
such that $fx=y$, there is a morphism $u\colon x\to x'$ in $\cal G_{0}$
such that $fu=v$.
\item For each object $x\in\cal G$, the map $\cal G\pr{x,x}\to\cal H\pr{fx,fx}$
is a Kan fibration.
\end{itemize}
\end{itemize}
Moreover, the $\W$-construction defines a right Quillen equivalence
$\W\colon \sf{Grpd}_{\Delta}\to\SS$.
\end{prop}
If simplicial sets model ``spaces'' and hence ``higher groupoids,''
then reduced simplicial sets should model ``higher groups.'' It
turns out that the $\W$-construction also substantiates this intuition.
\begin{prop}
[{\cite[Chapter II, \S 3, Theorem 2]{Quillen_HA}}]The category $\SG$
admits a model structure whose fibrations and weak equivalences are
created by the forgetful functor $\SG\to\SS$.
\end{prop}
\begin{prop}
[{\cite[Proposition 6.2]{GoerssJardine}}]The category $\SS_{0}$
admits a model structure whose cofibrations and weak equivalences
are created by the forgetful functor $\SS_{0}\to\SS$.
\end{prop}
\begin{thm}
[{\cite[Proposition 6.3]{GoerssJardine}}]The $\overline{W}$-construction
defines a right Quillen equivalence functor $\o W\colon \SG\to\SS_{0}$.
\end{thm}
As a corollary of the above theorem, we find that for any simplicial
group $G$, the simplicial set $\W G$ is a Kan complex. As explained
in \cite[Chapter V, \S4]{GoerssJardine}, the simplicial set $\W G$
is the base of a certain principal $G$-fibration $WG\to\W G$, and
the simplicial set $WG$ is contractible \cite[Chapter V, Lemma 4.6]{GoerssJardine}.
Combining this with the fact that every principal $G$-fibration is
a Kan fibration \cite[Corollary 2.7]{GoerssJardine}, we obtain:
\begin{cor}
For any simplicial group $G$, the simplicial set $\W G$ is a $BG$.
\end{cor}

\section{Main Result}

The goal of this section is to prove the main result of this note:
If $\cal G$ is a simplicial groupoid, there is a natural homotopy
equivalence
\[
\overline{W}\cal G\xrightarrow{\simeq}N\cal G.
\]

\begin{rem}
Note that $N\cal G$ is a Kan complex. Indeed, since simplicial groups
are Kan complexes, $\cal G$ is a fibrant simplicial category, and
so its homotopy coherent nerve $N\cal G$ is an $\infty$-category.
Moreover, its homotopy category $\ho\pr{N\cal G}\cong\pi_{0}\pr{\cal G}$
is a groupoid. Thus $N\cal G$ is an $\infty$-groupoid and hence
is a Kan complex.
\end{rem}
Our natural transformation $\overline{W}\to N$ is constructed from
a morphism of cosimplicial objects in the category $\Cat_{\Delta}$. We thus construct a cosimplicial object
corresponding to $\W$:

\begin{cst}[{\cite{Hinich07}}]

For each $n\geq0$, define $\Delta_{\W}^{n}$ to be the simplicial
category freely generated by an $\pr{n-i}$-arrow $g_{n,i}\colon i-1\to i$,
for each $1\leq i\leq n$. In other words, the objects of $\Delta_{\W}^{n}$
are the integers $0,\dots,n$, and the hom-simplicial sets are given
by
\begin{align*}
\Delta_{\W}^{n}\pr{i,j} & =\begin{cases}
\prod_{n-j\leq s<n-i}\Delta^{s} & \text{if }i\leq j,\\
\emptyset & \text{if }i>j.
\end{cases}
\end{align*}
We interpret the empty product as $\Delta^{0}$. The composition
map $\Delta_{\W}^{n}\pr{j,k}\times\Delta_{\W}^{n}\pr{i,j}\to\Delta_{\W}^{n}\pr{i,k}$
is the identity map.

If $\cal C$ is a simplicial category, a simplicial functor $f\colon\Delta_{\W}^{n}\to\cal C$
can be identified with a sequence $x_{0}\xrightarrow{f_{1}}x_{1}\xrightarrow{f_{2}}\cdots\xrightarrow{f_{n}}x_{n}$,
where $f_{i}$ is the image of the morphism $g_{n,i}$ under $f$.
We make $\{\Delta_{\W}^{n}\}_{n\geq0}$ into a cosimplicial object
in $\Cat_{\Delta}$ as follows: For $n\geq1$ and $0\leq i\leq n$,
the map $\partial_{i}\colon\Delta_{\W}^{n-1}\to\Delta_{\W}^{n}$ is given
by
\[
\partial_{i}\pr{g_{n-1,j}}=\begin{cases}
d_{i-j}g_{n,j} & \text{if }j<i\text{ or }i=n,\\
g_{n,i-1}\circ d_{0}g_{n,i} & \text{if }j=i<n,\\
g_{n,j+1} & \text{if }j>i.
\end{cases}
\]
For $n\geq0$ and $0\leq i\leq n$, the map $\sigma_{i}\colon\Delta_{\W}^{n+1}\to\Delta_{\o W}^{n}$
is given by
\[
\sigma_{i}\pr{g_{n+1,j}}=\begin{cases}
s_{i-j}g_{n,j} & \text{if }j\leq i,\\
\id_{i} & \text{if }j=i+1,\\
g_{n,j-1} & \text{if }j>i+1.
\end{cases}
\]
With this definition, the functor $\W\colon\sf{Grpd}_{\Delta}\to\SS$ is
the restriction of the functor $\Cat_{\Delta}\pr{\Delta_{\W}^{\bullet},-}\colon\Cat_{\Delta}\to\SS$.

\end{cst}
\begin{prop}
\label{prop:CtoG}There is a unique morphism
\[
\widetilde{\fr C}[\Delta^{\bullet}]\to\Delta_{\W}^{\bullet}
\]
of cosimplicial objects in $\Cat_{\Delta}$ such that each simplicial
functor $\widetilde{\fr C}[\Delta^{n}]\to\Delta_{\W}^{n}$ is the
identity on objects.
\end{prop}
\begin{rem}
If one wants to stick to the cosimplicial simplicial category $\fr C[\Delta^{\bullet}]$,
one needs to modify the $\W$-construction by replacing the mapping simplicial
sets of $\Delta_{\W}^{\bullet}$ by their opposites to obtain a corresponding
claim for Proposition \ref{prop:CtoG}.
\end{rem}
\begin{proof}[Proof of Proposition \ref{prop:CtoG}]
We begin by showing uniqueness. If there is a morphism $\varphi\colon\widetilde{\fr C}[\Delta^{\bullet}]\to\Delta_{\W}^{\bullet}$
of cosimplicial objects as in the statement, then the map $\varphi_{1}\colon\widetilde{\fr C}[\Delta^{1}]\to\Delta_{\W}^{1}$
must be the identity map since both $\widetilde{\fr C}[\Delta^{1}]$
and $\Delta_{\W}^{1}$ are isomorphic to the poset $[1]$. Since $\varphi$
commutes with the cosimplicial structure maps, it follows that for
each $0\leq i\leq j\leq n$, the image of $\{i,j\}\in\widetilde{\fr C}[\Delta^{n}]\pr{i,j}$
is completely determined. Since $\varphi$ commutes with compositions,
we find that $\varphi_{n}\colon\widetilde{\fr C}[\Delta^{n}]\pr{i,j}\to\Delta_{\W}^{n}\pr{i,j}$
is completely determined on vertices. But now this is a map between
the nerves of posets, and such a map is determined by its values on
vertices. This proves the uniqueness part.

Before moving on to the proof of the existence part, we prove the
following auxiliary assertion: Let $0\leq i\leq j\leq n$ be integers,
and let $\iota=\iota_{i,j}^{\pr n}\colon[1]\to[n]$ denote the poset map
defined by $\iota_{i,j}^{\pr n}\pr 0=i$ and $\iota_{i,j}^{\pr n}\pr 1=j$.
Then the map $\iota_{\ast}\colon\Delta_{\W}^{1}\pr{0,1}\to\Delta_{\W}^{n}\pr{i,j}=\prod_{i<k\leq j}\Delta^{n-k}$
maps the unique vertex of $\Delta_{\W}^{1}\pr{0,1}$ to the vertex
\[
\pr{0,1,\dots,j-i-1}\in\pr{\prod_{i<k\leq j}\Delta^{n-k}}_{0}.
\]
(When $i=j$, interpret the left hand side as the unique vertex $0\in\Delta_{0}^{0}=\Delta_{\W}^{n}\pr{i,j}$.)
The claim is proved by induction on $n$. The claim is trivial if
$n=0$. For the inductive step, suppose that the claim holds for $n-1$.
There are three cases to consider:

\begin{enumerate}[label=(\arabic*)]

\item If $i=j$, there is nothing to prove.

\item If $i<j<n$, then $\iota_{i,j}^{\pr n}=\partial_{n}\iota_{i,j}^{\pr{n-1}}$.
By definition, the map 
\[
\partial_{n}\colon\Delta_{\W}^{n-1}\pr{i,j}\to\Delta_{\W}^{n}\pr{i,j}
\]
 is given by
\[
\partial_{n-j}\times\cdots\times\partial_{n-i}\colon\Delta^{n-1-j}\times\cdots\times\Delta^{n-1-i}\to\Delta^{n-j}\times\cdots\times\Delta^{n-i}.
\]
So it fixes the vertex $\pr{0,1,\dots,j-i-1}$.

\item If $i<j=n$, then $\iota_{i,j}^{\pr n}=\partial_{n-1}\iota_{i,n-1}^{\pr{n-1}}$.
By definition, the map 
\[
\partial_{n-1}\colon\Delta_{\W}^{n-1}\pr{i,n-1}\to\Delta_{\W}^{n}\pr{i,n}
\]
 is given by
\[
\pr{\id,\partial_{0}}\times\partial_{1}\times\cdots\times\partial_{n-i}\colon\Delta^{0}\times\cdots\times\Delta^{n-1-i}\to\Delta^{0}\times\Delta^{1}\times\cdots\times\Delta^{n-i}.
\]
Thus it carries the vertex $\pr{0,1,\dots,n-i-2}$ to the vertex $\pr{0,1,\dots,n-i-1}$.

\end{enumerate}

We now proceed to the proof of the existence part. Define a simplicial
functor $\varphi_{n}\colon\widetilde{\fr C}[\Delta^{n}]\to\Delta_{\o W}^{n}$
as follows: Recall that $\widetilde{\fr C}[\Delta^{n}]\pr{i,j}$ is
the nerve of the \textit{opposite} of the poset 
\[
P_{i,j}=\{I\subset[i,j]\mid\min I=i,\,\max I=j\},
\]
with ordering given by inclusion. For $0\le i<j\leq n$, we
define a map of simplicial sets
\[
\varphi_{n}\colon\widetilde{\fr C}[\Delta^{n}]\pr{i,j}=N\pr{P_{i,j}^{\mathrm{op}}}\to\Delta_{\W}^{n}\pr{i,j}=\Delta^{n-j}\times\cdots\times\Delta^{n-i-1}
\]
on vertices to be the map induced by the poset map 
\begin{align*}
P_{i,j}^{\mathrm{op}} & \to[n-j]\times\cdots\times[n-i-1]\\
\{i=i_{0}<\dots<i_{k}=j\} & \mapsto\pr{0,\dots,i_{k}-i_{k-1}-1,\dots,0,\dots,i_{1}-i_{0}-1}.
\end{align*}
Note that this map is indeed a poset map because the ordering of $P_{i,j}^{\mathrm{op}}$
is given by the reverse inclusion. This defines a simplicial functor $\varphi_{n}\colon\widetilde{\fr C}[\Delta^{n}]\to\Delta_{\W}^{n}$.
We claim that the simplicial functors $\pr{\varphi_{n}}_{n\geq0}$
define a morphism of cosimplicial objects. In other words, we claim
that for any poset map $\alpha\colon[p]\to[q]$ and $0\le i\leq j\leq p$,
the diagram 
\[\begin{tikzcd}
	{\widetilde{\mathfrak{C}}[\Delta^p](i,j)} & {\Delta^{p}_{\overline{W}}(i,j)} \\
	{\widetilde{\mathfrak{C}}[\Delta^q](\alpha(i),\alpha(j))} & {\Delta^{q}_{\overline{W}}(\alpha(i),\alpha(j))}
	\arrow["{\varphi_p}", from=1-1, to=1-2]
	\arrow["{\alpha_\ast}"', from=1-1, to=2-1]
	\arrow["{\alpha_\ast}", from=1-2, to=2-2]
	\arrow["{\varphi_q}"', from=2-1, to=2-2]
\end{tikzcd}\]commutes. Since the simplicial sets in the diagrams are nerves of
posets, it suffices to show that the diagram commutes on the level
of vertices. Also, since $\varphi_{n},\alpha_{\ast}$ commutes with
compositions in $\widetilde{\fr C}[\Delta^{n}]$ and $\Delta_{\W}^{n}$,
it suffices to establish the identity 
\[
\alpha_{\ast}\varphi_{p}\pr{\{i,j\}}=\varphi_{q}\alpha_{\ast}\pr{\{i,j\}}.
\]
This is clear, because by construction, both sides are equal to $\iota_{\alpha\pr i,\alpha\pr j}^{\pr q}\pr{g_{1,1}}$.
The proof is now complete.
\end{proof}
We wish to show that the induced natural transformation $\W\to N\colon\sf{Grpd}_{\Delta}\to\SS$
is a natural weak equivalence. For this, the following proposition
comes in handy. 
\begin{prop}
The functor $N\colon\sf{Grpd}_{\Delta}\to\SS$ is right Quillen.
\end{prop}
\begin{proof}
We must show that $N$ is a right adjoint and preserves fibrations
and trivial fibrations. 

Let us begin by showing that $N$ is a right adjoint. According to
the adjoint functor theorem \cite[Theorem 1.66]{AR1994}, we only need
to show that $N$ preserves limits and filtered colimits, and that
both $\sf{sSet}$ and $\sf{Grpd}_{\Delta}$ are locally presentable.
The preservation of limits of follows from the fact that the inclusion
$\sf{Grpd}_{\Delta}\hookrightarrow\Cat_{\Delta}$ preserves limits
and $N\colon\Cat_{\Delta}\to\SS$ is a right adjoint. The preservation
of filtered colimits follows from the facts that the inclusion $\sf{Grpd}_{\Delta}\hookrightarrow\Cat_{\Delta}$
preserves filtered colimits, and that the simplicial categories $\widetilde{\fr C}[\Delta^{n}]$
are compact objects of $\Cat_{\Delta}$. For the local presentability,
let us recall the following facts:
\begin{enumerate}[label=(\arabic*)]
\item Every functor category of a locally presentable category is locally
presentable \cite[Corollary 1.54]{AR1994}.
\item A full subcategory of a locally presentable category closed under
limits and filtered colimits is again locally presentable \cite[Theorem 2.48]{AR1994}.
\end{enumerate}
It follows from (1) that $\SS$ is locally presentable, and combining
this with (2) shows that $\Cat$ and $\Cat^{\Del^{\op}}$ are locally
presentable. Another application of (2) to the inclusion $\sf{Grpd}_{\Delta}\hookrightarrow\Cat^{\Del^{\op}}$
shows that $\sf{Grpd}_{\Delta}$ is locally presentable. 

Next, to see that $N$ preserves weak equivalences, note that the
inclusion $\sf{Grpd}_{\Delta}\hookrightarrow\Cat_{\Delta}$ maps weak
equivalences in $\sf{Grpd}_{\Delta}$ to weak equivalences in $\Cat_{\Delta}$
between fibrant objects, because any simplicial group is a Kan complex.
Since weak categorical equivalences are weak homotopy equivalences,
it follows that $N$ preserves weak equivalences. 

It remains to verify that $N$ preserves fibrations. Since the homotopy
coherent nerve functor is a right Quillen functor from $\Cat_{\Delta}$
to $\SS_{\mathrm{Joyal}}$, and since the inclusion $\sf{Grpd}_{\Delta}\hookrightarrow\Cat_{\Delta}$
preserves fibrations, the functor $N$ maps fibrations in $\sf{Grpd}_{\Delta}$
to fibrations in the Joyal model structure. Now recall that $N\cal G$
is a Kan complex for any simplicial groupoid $\cal G$. By Joyal's
lifting theorem \cite[Theorem 2.1.8]{Landoo-cat}, every Joyal fibration
between Kan complexes is a Kan fibration. Thus $N$ preserves fibrations,
as claimed.
\end{proof}
We now arrive at the main result.
\begin{thm}
\label{thm:main}Let $\cal G$ be a simplicial groupoid. The morphism
$\widetilde{\fr C}[\Delta^{\bullet}]\to\Delta_{\W}^{\bullet}$ of cosimplicial
objects induces a homotopy equivalence
\[
\W\cal G\xrightarrow{\simeq}N\cal G
\]
of Kan complexes.
\end{thm}
\begin{proof}
The natural transformation $\W\to N\colon\sf{Grpd}_{\Delta}\to\SS$ induces
a natural transformation $\bb R\W\to\bb RN$ between the total right
derived functors. We must show that the latter natural transformation
is a natural isomorphism. By the uniqueness of adjoints, it suffices
to show that the induced natural transformation $\bb LL_{N}\to\bb LL_{\o W}$
is a natural isomorphism, where $L_{\W},L_{N}\colon\SS\to\sf{Grpd}_{\Delta}$
are the left adjoints of $\W$ and $N$. 

We begin by showing that $L_{N}\pr{\Delta^{n}}\to L_{\W}\pr{\Delta^{n}}$
is a weak equivalence for every $n\geq0$. Since the map $\Delta^{n}\to\Delta^{0}$
is a weak homotopy equivalence, it suffices to prove this for $n=0$.
But now we have natural bijections
\[
\sf{Grpd}_{\Delta}\pr{L_{N}\pr{\Delta^{0}},\cal G}\cong\SS\pr{\Delta^{0},N\cal G}\cong\opn{ob}\cal G
\]
and
\[
\sf{Grpd}_{\Delta}\pr{L_{\W}\pr{\Delta^{0}},\cal G}\cong\SS_{0}\pr{\Delta^{0},\W\cal G}\cong\opn{ob}\cal G,
\]
which shows that $L_{N}\pr{\Delta^{0}}$ and $L_{\o W}\pr{\Delta^{0}}$
are the terminal simplicial groupoids. Thus the claim holds trivially.

Now let $X$ be an arbitrary simplicial set. We show that $L_{N}\pr X\to L_{\W}\pr X$
is a weak equivalence of simplicial groups. Since $X$ is a colimit
of the sequence of cofibrations between cofibrant objects
\[
\sk_{0}X\to\sk_{1}X\to\cdots,
\]
it suffices to consider the case where $X$ is isomorphic to its $n$-skeleton.
We prove the claim by induction on $n$. The base case $n=0$ follows
from the result in the previous paragraph. For the inductive step,
assume the claim holds for $n$. We have a pushout diagram of the
form 
\[\begin{tikzcd}
	{\coprod_{\alpha }\partial \Delta^{n+1}} & {\operatorname{sk}_nX} \\
	{\coprod_{\alpha }\Delta^{n+1}} & X,
	\arrow[from=1-1, to=2-1]
	\arrow[from=1-1, to=1-2]
	\arrow[from=1-2, to=2-2]
	\arrow[from=2-1, to=2-2]
\end{tikzcd}\]and by the induction hypothesis, the claim holds for all the corners
except for $X$. Now observe that the image of the above square under
any left Quillen functor is a homotopy pushout, because all the relevant
objects are cofibrant and the left vertical arrow is a cofibration.
Hence $L_{N}X\to L_{\W}X$ is a weak equivalence, as required.
\end{proof}
As a corollary, we find that the homotopy coherent nerve models the
classifying space:
\begin{cor}
\label{cor:main}For any simplicial group $G$, the map
\[
\W G\to NG
\]
is a homotopy equivalence of Kan complexes.
\end{cor}

\providecommand{\bysame}{\leavevmode\hbox to3em{\hrulefill}\thinspace}
\providecommand{\MR}{\relax\ifhmode\unskip\space\fi MR }
\providecommand{\MRhref}[2]{%
  \href{http://www.ams.org/mathscinet-getitem?mr=#1}{#2}
}
\providecommand{\href}[2]{#2}

\end{document}